\documentclass[12pt,reqno,a4paper]{amsart}
\usepackage{amssymb,latexsym}
\usepackage{amsthm, amsmath, amsfonts}
\usepackage{hyperref}
\usepackage{empheq}
\usepackage{chngcntr}
\usepackage{apptools}
\usepackage{graphicx}  
\usepackage{dcolumn}   
\usepackage{bm}        
\usepackage{enumerate}
\usepackage{slashed}
\usepackage{simplewick}
\usepackage{comment}
\usepackage{color}
\usepackage{soul}
\usepackage[english]{babel}
\usepackage{appendix}
\usepackage[utf8]{inputenc}
\usepackage{tikz} 
\usetikzlibrary{decorations.markings}
\usepackage{enumitem}
\numberwithin{equation}{section}

\usepackage{marginnote}
\usepackage[normalem]{ulem}

\newtheorem*{maintheo}{Theorem}

\newtheorem{theorem}{Theorem}[section]
\newtheorem{lemma}[theorem]{Lemma}
\newtheorem{proposition}[theorem]{Proposition}

\theoremstyle{definition}
\newtheorem{definition}[theorem]{Definition}

\newtheorem*{condition}{Condition $\mathfrak{A}$}
\theoremstyle{remark}
\newtheorem{remark}[theorem]{Remark}
\newtheorem{example}[theorem]{Example}

\title[LDP for DPP on Stein manifolds]{Large deviation principle for determinantal point processes on Stein manifolds}
\author{Kiyoon Eum}
\address{Department of Mathematical Sciences, KAIST, 291 Daehak-ro, Yuseong-gu, Daejeon 34141, South Korea}
\email{kyeum@kaist.ac.kr}

\begin{document}

\begin{abstract}
In this paper, we study natural determinantal point processes on Stein manifolds defined by canonical sections and plurisubharmonic weight functions. We present a natural condition on the weights under which these processes satisfy a large deviation principle with a good rate function given by the Fenchel--Legendre transform of the pluripotential equilibrium energy. Along the way, we prove some pluripotential theoretic results on Stein manifolds, such as the $C^{1,1}$ regularity of appropriate envelopes. These results are noncompact analogues of results of Berman and Berman--Boucksom on compact K\"ahler manifolds.
\end{abstract}

\maketitle

\thispagestyle{empty}
\tableofcontents


\section{Introduction}

Determinantal point processes are point processes whose correlation functions are given by determinants of a kernel. Introduced by Macchi \cite{macchi1975coincidence}, they arise in random matrix theory and in the description of fermionic systems; see \cite{soshnikov2000determinantal,shirai2003random,hough2009zeros}. Bergman kernels provide a natural geometric source of such processes. A Hilbert space of square-integrable holomorphic sections determines an orthogonal projection, and its kernel defines a determinantal process whenever the projection is locally trace class. The asymptotic distribution of the points is then closely related to the geometry of the weight defining the Hilbert space.

For powers of a line bundle over a compact complex manifold, Berman established a large deviation principle (LDP) for the associated determinantal processes \cite{berman2014determinantal}. The limiting measure is a pluripotential equilibrium measure, and the rate function is expressed through the dual of the equilibrium energy. The asymptotics of volumes of balls of holomorphic sections and the differentiability of this energy, proved by Berman and Boucksom \cite{berman2010growth}, provide the main complex geometric input for this result. Berman's results also include ensembles on noncompact supports, but those ensembles still arise from finite-dimensional spaces of sections and have finitely many particles for each value of the parameter. Large deviations are also known for particular observables of infinite processes: Shirai proved an LDP for rescaled point counts in expanding disks for the infinite fermion process \cite{shirai2006large}. Our concern here is a measure-valued LDP for infinite rank Bergman processes on Stein manifolds, directly generalizing Berman's result.

In \cite{eum2025limiting}, we studied weighted Bergman determinantal processes on bounded pseudoconvex domains in $\mathbb{C}^n$ and obtained a limit for their scaled cumulant generating functions. The Bergman spaces in that setting are infinite dimensional, and Fredholm determinants replace the finite dimensional determinants used in the compact case. The result applies to compactly supported smooth perturbations $u$ for which the perturbed weight $\phi+u$ remains strictly plurisubharmonic. Extending it to arbitrary compactly supported perturbations requires a plurisubharmonic envelope. As discussed in \cite[Remark 2]{eum2025limiting}, the difficulty is that this envelope need not differ from $\phi$ by a continuous compactly supported function. The finiteness and integration-by-parts arguments for the relative energy therefore require additional control at infinity.

In this paper, we obtain such control from a sufficient geometric growth condition on the reference weight. Let $X$ be a Stein manifold of complex dimension $n\geq1$. A motivating example is furnished by a proper holomorphic embedding $F:X\hookrightarrow\mathbb{C}^N$. The function $\phi_F=1+|F|^2$ is a smooth strictly plurisubharmonic exhaustion, and $i\partial\bar\partial\phi_F$ is the pullback of the Euclidean K\"ahler form. We consider it as an analogue of Fubini--Study potential in Stein setting. The Cauchy--Schwarz inequality shows that $\phi_F$ satisfies the following condition with $B=1$.
\begin{condition}
Let $\phi$ be a smooth strictly plurisubharmonic exhaustion function on $X$. We say that $\phi$ satisfies \textit{Condition $\mathfrak{A}$} if
\begin{equation}\label{conditionA}
    \phi \geq 1 \quad \text{and}\quad i\partial\phi\wedge\bar{\partial}\phi \leq B \phi \,i\partial\bar{\partial}\phi
\end{equation}
for some $B<\infty$.
\end{condition}

Thus every Stein manifold admits a weight satisfying Condition $\mathfrak{A}$. The lower bound $\phi\geq1$ is a convenient normalization and it can be arranged by adding a constant to a nonnegative weight satisfying the inequality condition. Related growth conditions on K\"ahler potentials occur in McNeal's study of $L^2$ harmonic forms \cite{mcneal20022}. In our setting, Condition $\mathfrak{A}$ implies the completeness of the K\"ahler metric $\omega_\phi=i\partial\bar\partial\phi$ and provides the control needed to localize the envelope problem.

For a real-valued $u\in C_c^\infty(X)$, let
\begin{equation*}
    P(\phi+u):=\left(\sup\{ v \in PSH(X):v \leq \phi+u\}\right)^*,
\end{equation*}
where the star denotes upper semicontinuous regularization. This is the envelope over all plurisubharmonic functions on $X$. In the weighted polynomial setting on $\mathbb{C}^n$ of \cite{berman2009poly}, the spaces of polynomials of bounded degree and an envelope in the Lelong class were used. Here we use the full weighted Bergman space, and Condition $\mathfrak{A}$ allows us to work with the ordinary envelope above. More precisely, we prove that
\begin{equation*}
    P(\phi+u)-\phi\in C_c^{1,1}(X)
\end{equation*}
for every $u\in C_c^\infty(X)$ (Theorem \ref{plurimain}). In particular, the envelope equals $\phi$ outside a compact set. The regularity argument adapts the method of Bedford--Taylor \cite{MR445006}, using holomorphic vector fields supplied by Cartan's Theorem A.

We now introduce the objects in the main theorem. Let $K_X$ denote the canonical bundle of $X$ and set $c_n=2^{-n}i^{n^2}$. For $k>0$, consider
\begin{equation*}
    \mathcal H(k\phi)=
    \left\{s\in H^0(X,K_X):
    \|s\|_{k\phi}^2:=\int_X c_n\,s\wedge\overline{s}\,e^{-k\phi}<\infty\right\}.
\end{equation*}
The use of holomorphic $n$-forms makes this norm independent of an auxiliary volume form. The corresponding Bergman projection defines a canonical determinantal process $\Theta_k$, as explained in Definition \ref{definitionDPP}. Under Condition $\mathfrak{A}$, $\mathcal H(k\phi)$ is infinite dimensional for every $k>0$ (Remark \ref{infinitedimension}). We consider the rescaled counting measures as elements of the space of positive Radon measures with the vague topology, that is, convergence against functions in $C_c(X)$. 

We use the convention $dd^c=\frac{i}{2\pi}\partial\bar\partial$ and write $\operatorname{MA}(a)=(dd^c a)^n/n!$. For a plurisubharmonic function $a$ with $a-\phi\in C_c^{1,1}(X)$, define the relative Monge--Amp\`ere energy by
\begin{equation*}
    \mathcal E_\phi(a)=
    \frac{1}{(n+1)!}\sum_{j=0}^n
    \int_X (a-\phi)(dd^c a)^j\wedge(dd^c\phi)^{n-j}.
\end{equation*}
These integrals are finite because $a-\phi$ is compactly supported, and $\mathcal E_\phi(\phi)=0$. The envelope result above makes the equilibrium energy $\mathcal E_\phi(P(\phi+u))$ well-defined for every smooth compactly supported $u$. Moreover, we will show that it is differentiable (Theorem \ref{diffofequil}), following \cite{berman2010growth}. All test functions below are real-valued.

\begin{maintheo}[=Theorem \ref{thmmain}]
Let $X$ be a Stein manifold with a plurisubharmonic exhaustion function $\phi$ satisfying the Condition $\mathfrak{A}$. Let $\Theta_k$ be the determinantal point process associated with the weighted Bergman space of holomorphic $n$-forms with weight $e^{-k\phi}$, defined in Definition \ref{definitionDPP}. Set 
\begin{equation*}
    L_k=\frac{1}{k^n}\Theta_k.
\end{equation*}
The laws of $L_k$, on the space of positive Radon measures with the vague topology, satisfy a large deviation principle with speed $r_k=k^{n+1}$ and a good rate function
\begin{equation*}
    I_\phi(\mu)=\sup_{u\in C_c^\infty(X)}\biggl\{ \mathcal{E}_\phi(P(\phi+u))-\int_X ud\mu \biggr\}.
\end{equation*}
The unique zero of $I_\phi$ is
\begin{equation*}
    \mu_\phi=\operatorname{MA}(\phi)=\frac{(dd^c \phi)^n}{n!}.
\end{equation*}
\end{maintheo}

\begin{example}
Let $X=\mathbb{C}^n$ and $\phi_0(z)=1+|z|^2$. Then $\phi_0$ satisfies Condition $\mathfrak{A}$ with $B=1$. Identifying a holomorphic $n$-form with its coefficient relative to $dz^1\wedge\cdots\wedge dz^n$, the space $\mathcal{H}(k\phi_0)$ is a weighted Fock space. After absorbing the Gaussian weight into the kernel, the associated determinantal process $\Theta_k$ has correlation kernel
\begin{equation*}
    \widetilde{K}_k(z,w)=\left(\frac{k}{\pi}\right)^n \exp\left(k z\cdot\overline{w} -\frac{k}{2}|z|^2 -\frac{k}{2}|w|^2 \right)
\end{equation*}
with respect to Lebesgue measure $d\lambda_{2n}$, where $z\cdot\overline{w}=\sum_{j=1}^n z^j\overline{w^j}$. In particular,
\begin{equation*}
    \frac{1}{k^n}\beta_{k\phi_0}=\frac{1}{\pi^n}d\lambda_{2n}=\operatorname{MA}(\phi_0)
\end{equation*}
for every $k>0$. The main theorem therefore gives an LDP for $L_k=k^{-n}\Theta_k$ in the vague topology, with speed $k^{n+1}$ and good rate function $I_{\phi_0}$, whose unique zero is $\pi^{-n}d\lambda_{2n}$.

When $n=1$, $\Theta_1$ is the standard infinite fermion process with exponential kernel \cite{shirai2006large}. If $\mathcal{G}$ denotes its random configuration, then
\begin{equation*}
    \Theta_k \overset{\mathrm{law}}{=} \sum_{\zeta\in\mathcal{G}} \delta_{\zeta/\sqrt{k}}.
\end{equation*}
Equivalently, writing $k=R^2$, the random measures
\begin{equation*}
    \frac{1}{R^2}\sum_{\zeta\in\mathcal{G}} \delta_{\zeta/R}
\end{equation*}
satisfy an LDP as $R\rightarrow\infty$, with speed $R^4$ and unique zero $\pi^{-1}dA$, where $dA$ is planar Lebesgue measure.
\end{example}

Section 2 collects the complex geometric and probabilistic preliminaries. In Section 3, we construct the canonical determinantal point process, establish the envelope and equilibrium energy results, and prove the Bergman kernel asymptotics. We then derive the limit of the cumulant generating functions and complete the proof of the large deviation principle using the abstract G\"artner--Ellis theorem.

\section{Preliminaries}

\subsection{Preliminaries on complex geometry}\label{prelimcplx}

In this subsection we introduce the setting of our DPP, namely Stein manifolds, weighted Bergman spaces of holomorphic $n$-forms, and their associated Bergman kernels. For a reference on Stein manifolds, see \cite[Chapter 2]{forstnerivc2011stein}. We also recall Demailly's singular $L^2$-estimate on a complete K\"ahler manifold which we will use in the proof of the main result. 

\begin{definition}
A complex manifold $X$ is called a \textit{Stein manifold} if

(1) For any compact subset $K\subset X$, $\widehat{K}:=\{x\in X : |f(x)|\leq \sup_K |f|, \;\forall f\in \mathcal{O}(X)\}$ is compact.

(2) For any $x\neq y\in X$,  there exists $f\in \mathcal{O}(X)$ such that $f(x)\neq f(y)$. 
\end{definition}
Equivalently, Stein manifolds can be characterized by the existence of strictly plurisubharmonic exhaustion functions. Two theorems about Stein manifolds are important for us. The first one is Cartan's Theorem A, which implies that any holomorphic vector bundle on Stein manifolds, in particular holomorphic vector bundle, is globally generated. The other one is the embedding theorem due to Remmert, stating that any Stein manifold admits a proper holomorphic embedding into $\mathbb{C}^N$ for some $N\in\mathbb{N}$. From this we have the following existence result for psh exhaustion functions satisfying Condition $\mathfrak{A}$.

\begin{proposition}
Let $X$ be a Stein manifold. Then there is a function $\phi$ satisfying Condition $\mathfrak{A}$. For example, let $F=(F_1,\cdots,F_N):X\hookrightarrow \mathbb{C}^N$ be a proper holomorphic embedding. Then
\begin{equation*}
    \phi_F=1+\sum_{j=1}^N |F_j|^2
\end{equation*}
satisfies the Condition $\mathfrak{A}$ with $B=1$.
\end{proposition}

\begin{proof}
Properness of $\phi_F$ follows from properness of $F$. Since $F$ is an immersion, 
\begin{equation*}
    i\partial\bar{\partial} \phi_F = i \sum_j \partial F_j \wedge \bar{\partial}\overline{F_j}=i \sum_j d F_j \wedge d \overline{F_j}=F^*(i\sum_j dz_j\wedge d\bar{z}_j)>0.
\end{equation*}
Hence $\phi_F$ is a strictly psh exhaustion function. Equation \eqref{conditionA} with $B=1$ follows from the Cauchy--Schwarz inequality: for $\xi\in T^{1,0}_x X$,
\begin{equation*}
    |\partial \phi_F(\xi)|^2=\left\lvert\sum_j \overline{F_j(x)}dF_j(\xi)\right\rvert^2 \leq \sum_j |F_j(x)|^2 \,\sum_j |dF_j(\xi)|^2\leq \phi_F(x)\partial\bar{\partial}\phi_F(\xi,\bar\xi).
\end{equation*}
Alternatively, let $\psi$ be any strictly psh exhaustion function and let $\phi:=e^{\psi-\inf \psi}$. It is easy to check that $\phi$ satisfies Condition $\mathfrak{A}$. Indeed, if $X=\mathbb{C}^n$, $\phi_{\mathrm{id}}=1+|z|^2=e^{\log(1+|z|^2)}=e^{\phi_{FS}}$.
\end{proof}

The following Proposition provides a complete K\"ahler metric which will be exploited in an $L^2$ estimate later, from psh weight $\phi$ with condition $\mathfrak{A}$.

\begin{proposition}\label{complete}
Let $\phi$ satisfy the condition $\mathfrak{A}$. Then $\omega_\phi:=i\partial\bar\partial \phi$ is a complete K\"ahler metric.
\end{proposition}

\begin{proof}
The condition \eqref{conditionA} implies $|\nabla_{\omega_\phi} \phi|^2_{\omega_\phi}\leq C_B\phi$ for some constant $C_B$ only depending on $B$. In particular, this implies $|\nabla_{\omega_\phi}\sqrt{\phi}|_{\omega_\phi}^2\leq C_B/4$ ($\sqrt{\phi}$ is smooth since $\phi\geq1$). Since $\sqrt{\phi}$ is proper, every closed $\omega_\phi$-ball is compact.
\end{proof}

Let $K_X=\bigwedge^n (T^{1,0}X)^*$ be the canonical bundle on $X$. Set $c_n:=2^{-n}i^{n^2}$. For a weight function $\phi$ (not necessarily psh), we consider the weighted Bergman space
\begin{equation}
    \mathcal{H}(k\phi):=\Bigl\{ s\in H^0(X,K_X) : \|s\|^2_{k\phi}:=\int_X c_n\,s\wedge \overline{s}e^{-k\phi}<\infty \Bigr\}.
\end{equation}
If $a$ is any other continuous weight such that $\phi-a$ is compactly supported, the two $L^2$ norms are equivalent and the underlying vector spaces $\mathcal{H}(ka), \mathcal{H}(k\phi)$ are the same. We will see later that if $\phi$ satisfies Condition $\mathfrak{A}$, they are infinite dimensional (Remark \ref{infinitedimension}). Note that so far we do not impose a hermitian metric on $K_X$. Now choose an arbitrary smooth positive volume form $dV$ on $X$. This induces a hermitian metric $|\cdot|_{dV}$ on $K_X$ by
\begin{equation*}
    |s|_{dV}^2dV=c_n\,s\wedge \overline{s}.
\end{equation*}
Under this hermitian metric, $\mathcal{H}(k\phi)$ can be considered as a subspace of $L^2(X,K_X,|\cdot|_{dV}e^{-k\phi}dV)$. By the submean value inequality, $\mathcal{H}(k\phi)$ is a reproducing kernel Hilbert space. The reproducing kernel, called the \textit{weighted Bergman kernel}, of this space is given by
\begin{equation*}
    K_{k\phi}(x,y)=\sum_j s_j(x)\otimes \overline{s_j}(y) \in K_{X,x}\boxtimes \overline{K_{X,y}},
\end{equation*}
where the $s_j$ form an orthonormal basis of $\mathcal{H}(k\phi)$. Also, define the \textit{Bergman measure} by
\begin{equation}\label{Bergmanmeasure}
    \beta_{k\phi}=\sum_j c_n\, s_j\wedge \overline{s_j} e^{-k\phi}.
\end{equation}
Choose local coordinates $z=(z^1,\cdots,z^n)$. Denote $dz^1\wedge\cdots \wedge dz^n$ by $dZ$. Then locally, $s_j=f_jdZ$ for some holomorphic functions $f_j$ and
\begin{equation*}
    K_{k\phi}(z,z)=\sum_j |f_j(z)|^2 dZ\otimes \overline{dZ}, \quad \beta_{k\phi}=c_n\sum_j |f_j(z)|^2e^{-k\phi(z)} dZ \wedge d\bar{Z}.
\end{equation*}
It is well known that the diagonal of the Bergman kernel has the following extremal characterization:
\begin{equation}\label{Bk}
    B_{k\phi}(z):=\frac{d\beta_{k\phi}}{c_n dZ\wedge d\bar{Z}}=\sum_j |f_j(z)|^2e^{-k\phi(z)}=\sup_{0\neq s\in\mathcal{H}(k\phi)}\frac{|f_s(z)|^2e^{-k\phi(z)}}{\|s\|^2_{k\phi}},
\end{equation}
where $s=f_s \,dZ$. Note that $B_{k\phi}$ here only makes sense once we choose a local chart.

To readily apply existing results on determinantal point processes in Section \ref{section3.1}, which are usually formulated in terms of functions, we need to translate these constructions to the space of functions. Choose a global Borel measurable locally bounded unitary frame $e$ for $K_X$ with the metric given by $dV$. The section $e$ does not need to be continuous or holomorphic. If locally $dV=c_n v(z) dZ\wedge d\bar{Z}$, $e$ is locally expressed as $e=\sqrt{v}dZ$. Define the unitary map
\begin{align*}
    U_e : L^2(X,K_X,|\cdot|_{dV}e^{-k\phi}dV)&\longrightarrow L^2(X,e^{-k\phi}dV)\\
    s= ae \quad &\longmapsto \quad a.
\end{align*}
Denoting by $P_{k\phi} : L^2(X,K_X,|\cdot|_{dV}e^{-k\phi}dV)\longrightarrow \mathcal{H}(k\phi)$ the orthogonal projection onto the subspace of holomorphic sections, we obtain that $U_e \circ P_{k\phi} \circ U_e^{-1}:L^2(X,e^{-k\phi}dV) \longrightarrow U_e(\mathcal{H}(k\phi))$ is an orthogonal projection with a well-defined scalar integral kernel
\begin{equation*}
    K^{dV,e}_{k\phi}(x,y)=\sum_j a_j(x) \overline{a_j}(y)
\end{equation*}
where $a_j=U_e(s_j)$. In general, $K^{dV}_{k\phi}$ is merely locally bounded. If we choose another unitary frame $e'=ge$ with $|g|=1$, $g\times U_{e'}=U_e$ and
\begin{equation}\label{choice1}
    g(x)\overline{g}(y)K^{dV,e'}_{k\phi}(x,y)=K^{dV,e}_{k\phi}(x,y) \implies \det \left[K^{dV,e'}_{k\phi}(x_i,x_j)\right]_{i,j}=\det \left[K^{dV,e}_{k\phi}(x_i,x_j)\right]_{i,j}.
\end{equation}
Also, if we choose a different volume form $dV'=hdV$, $e'=\sqrt{h}e$ is a unitary frame for this new hermitian metric, and we have
\begin{equation}\label{choice2}
    \det \left[K^{dV',e'}_{k\phi}(x_i,x_j)\right]_{i,j}\prod_{j=1}^m dV'(x_j)=\det \left[K^{dV,e}_{k\phi}(x_i,x_j)\right]_{i,j}\prod_{j=1}^m dV(x_j).
\end{equation}
These considerations will be used to define a canonical determinantal point process on $X$ that only depends on the choice of $\phi$, not $dV$ or $e$.

Finally, we will need the following special case of an $L^2$ estimate due to Demailly.
\begin{theorem}[\protect{\cite[Theorem 5.1, Remark 4.5]{Demailly1982}}]\label{demailly}
Let $(X,\omega)$ be a complete K\"ahler manifold. Let $\varphi$ be an upper semicontinuous function on $X$ which is locally the sum of a smooth function and a plurisubharmonic function. Suppose that $i\partial\bar\partial \varphi \geq \delta \omega$ for some constant $\delta>0$. Then, for every continuous $(n,1)$-form $g$ such that
\begin{equation*}
    \bar\partial g=0,\qquad \int_X |g|_\omega^2\,e^{-\varphi}\frac{\omega^n}{n!}<\infty,
\end{equation*}
there exists a $(n,0)$-form $f$ such that 
\begin{equation*}
    \bar\partial f=g,\qquad \int_X c_n\,f\wedge \overline{f}\;e^{-\varphi} \leq \frac{C}{\delta}\int_X |g|_\omega^2\,e^{-\varphi}\frac{\omega^n}{n!},
\end{equation*}
for some constant $C$ only depending on $n$. 
\end{theorem}

\subsection{Preliminaries on probability}

Here we explain the notion of a large deviation principle and record some measure theoretic backgrounds and notations for our discussion.

Let $(\mathcal{X},\mathcal{B}_{\mathcal{X}})$ be a topological space with Borel $\sigma$-field. A \textit{rate function} $I$ is a lower semicontinuous mapping $I : \mathcal{X}\longrightarrow [0,\infty]$. A \textit{good rate function} is a rate function for which all the sublevel sets $I^{-1}([0,\alpha])$ are compact subsets of $\mathcal{X}$. For any set $\Gamma$, $\overline{\Gamma}$ denotes the closure of $\Gamma$, $\Gamma^\circ$ the interior of $\Gamma$. Let $\{\mu_k\}_{k>0}$ be a family of probability measures on $(\mathcal{X},\mathcal{B}_{\mathcal{X}})$.
\begin{definition}
   A family of probability measures $\{\mu_k\}$ is said to satisfy \textit{the large deviation principle (LDP)} with a rate function $I$ and speed $r_k$ if, for all $\Gamma\in \mathcal{B}_{\mathcal{X}}$,
\begin{equation*}
    -\inf_{x\in \Gamma^\circ}I(x)\leq \liminf_{k\rightarrow\infty}\frac{1}{r_k}\log \mu_k(\Gamma)\leq \limsup_{k\rightarrow \infty}\frac{1}{r_k}\log \mu_k(\Gamma)\leq -\inf_{x\in\overline{\Gamma}}I(x).
\end{equation*}
\end{definition}
That is, as $k\rightarrow\infty$, the probability measures $\mu_k$ concentrate near the minimum of $I$ exponentially: $\mu_k \approx e^{-r_k I}$. We also need a closely related concept called exponential tightness.
\begin{definition}
   A family of probability measures $\{\mu_k\}$ is said to be \textit{exponentially tight} at speed $r_k$ if for every $\alpha<\infty$, there exists a compact set $\mathcal{K}_\alpha \subset \mathcal{X}$ such that
\begin{equation*}
\limsup_{k\rightarrow\infty}\frac{1}{r_k}\log \mu_k(\mathcal{X}\setminus \mathcal{K}_\alpha)\leq -\alpha.
\end{equation*}
\end{definition}

Our proof of the LDP will be based on the following abstract G\"artner--Ellis theorem.
\begin{theorem}[Abstract G\"artner--Ellis theorem \protect{\cite[corollary 4.6.14]{dembo2009large}}]\label{thmGE}
Let $\{\mu_k\}$ be an exponentially tight family of Borel probability measures on the locally convex Hausdorff topological vector space $\mathcal{X}$. Define the logarithmic moment generating function $\Lambda_k: \mathcal{X}^*\longrightarrow (-\infty,\infty]$ by
\begin{equation*}
    \Lambda_k(\ell)=\log \mathbb{E}_{\mu_k}[e^{\langle \ell, \cdot \rangle}]=\log \int_\mathcal{X} e^{\ell(x)}\mu_k(dx),\quad \ell\in\mathcal{X}^*.
\end{equation*}
Suppose $\Lambda(\cdot)=\lim_{k\rightarrow\infty}\frac{1}{r_k}\Lambda_k(r_k\cdot)$ is finite valued and Gateaux differentiable. Then $\{\mu_k\}$ satisfies the LDP in $\mathcal{X}$ with the convex, good rate function $\Lambda^*$. Here $\Lambda^*$ denotes the Fenchel--Legendre transform
\begin{equation*}
    \Lambda^*(x):=\sup_{\ell\in\mathcal{X}^*}\{\ell(x)-\Lambda(\ell)\},\quad x\in\mathcal{X}.
\end{equation*}
\end{theorem}

Now let $X$ be a complex manifold. We will be working in the following setting. Let $M:=\mathcal{M}_{\mathrm{loc}}(X,\mathbb{R})$ be the real vector space of signed Radon measures with locally finite total variation. Endow $M$ with the weak topology $\tau$ generated by $C^\infty_c(X,\mathbb{R})$. That is, $(M,\tau)^*=C^\infty_c(X,\mathbb{R})$. Let $M_+:=\mathcal{M}^+_{\mathrm{loc}}(X,\mathbb{R})$ be the closed cone of positive measures. Then it is easy to show that the subspace topology on $M_+$ induced from $\tau$ is identical to the \textit{vague topology} on $M_+$, generated by $C_c(X,\mathbb{R})$. This can be shown by smooth locally uniform approximation of test functions and a local mass bound for elements of $M_+$. The following theorem gives a useful criterion for compactness in the vague topology.
\begin{theorem}[\protect{\cite[Theorem A2.3]{MR1876169}}]\label{thmcompact}
A set $A\subset M_+$ is relatively compact in vague topology if and only if $\sup_{\mu\in A}\int_X f\mu <\infty$ for every nonnegative $f\in C_c(X)$.
\end{theorem}

\section{Main results}

\subsection{Definition of DPP on Stein manifolds}\label{section3.1}

For the basics of determinantal point processes, see \cite{hough2009zeros}. Here we recall the existence result of Macchi and Soshnikov.
\begin{theorem}[\protect{\cite{macchi1975coincidence,soshnikov2000determinantal}}]\label{MacSo}
Let $\Omega$ be a locally compact Polish space, $\mu$ a Radon measure on $\Omega$, and $K:\Omega\times\Omega\longrightarrow \mathbb{C}$ a locally integrable function. Suppose $K$ determines a self-adjoint integral operator $\mathbb{K}$ on $L^2(\Omega,\mu)$ that is locally trace class. Then $K$ defines a determinantal process $\Theta$ on $\Omega$ if and only if the spectrum of $\mathbb{K}$ is contained in $[0,1]$. That is, for any family of mutually disjoint subsets $D_1,\cdots,D_m$ of $\Omega$,
\begin{equation*}
    \mathbb{E}\left[\prod_{j=1}^m \Theta (D_j)\right]=\int_{\prod_j D_j} \det \left[K(x_i,x_j)\right]_{i,j=1,\cdots,m} d\mu(x_1)\cdots d\mu(x_m).
\end{equation*}
\end{theorem}

Let $X$ be a Stein manifold and $\phi$ be a psh exhaustion function satisfying Condition $\mathfrak{A}$. Choose any smooth volume form $dV$ and a corresponding locally integrable global unitary frame $e$ of $K_X$ as in Section \ref{prelimcplx}. To apply the Macchi--Soshnikov theorem \ref{MacSo} to $K^{dV,e}_{k\phi}$, we need to show that $U_e \circ P_{k\phi} \circ U_e^{-1}$ is, or equivalently, $P_{k\phi}$ is locally trace class. We also record the formula for the trace of Toeplitz operators for later use.
\begin{lemma}\label{lemma1}
    For every compact set $K\subset X$, the restricted Bergman projection on $K$ is trace class and
\begin{equation*}
    \operatorname{Tr}(\mathbf{1}_K P_{k\phi}\mathbf{1}_K)=\int_K \beta_{k\phi}<\infty.
\end{equation*}
More generally, for a bounded function $g$ on $X$, define a Toeplitz operator $T_{k\phi,g}$ with symbol $g$ by
\begin{equation*}
    T_{k\phi,g}:=P_{k\phi} M_g |_{\mathcal{H}(k\phi)},
\end{equation*}
where $M_g$ is a multiplication operator by $g$. If $g$ is compactly supported continuous, then $T_{k\phi,g}$ is trace class and
\begin{equation}\label{traceformula}
    \operatorname{Tr}(T_{k\phi,g})=\int_X g\,\beta_{k\phi}.
\end{equation}
\end{lemma}
\begin{proof}
Define $A_K:=\mathbf{1}_KP_{k\phi}:L^2\longrightarrow L^2$. Choose an orthonormal basis $s_j$ of $\mathcal{H}(k\phi)$ and extend it to an orthonormal basis of $L^2$. Then $A_K$ is a Hilbert--Schmidt operator by
\begin{equation*}
    \|A_K \|^2_2=\sum_j \|\mathbf{1}_K s_j\|^2_{k\phi}=\sum_j\int_X \mathbf{1}_K c_n\,s_j\wedge \overline{s_j}\,e^{-k\phi}=\int_K \beta_{k\phi}<\infty.
\end{equation*}
Hence we have
\begin{equation*}
    \operatorname{Tr}(\mathbf{1}_K P_{k\phi}\mathbf{1}_K)=\operatorname{Tr}(A_K A_K^*)= \|A_K \|^2_2=\int_K \beta_{k\phi}.
\end{equation*}
The Toeplitz operator case can be proved similarly.
\end{proof}

With these results, we can define a determinantal point process on a Stein manifold that only depends on the choice of $\phi$.
\begin{definition}[DPP $\Theta_k$ on $(X,\phi)$]\label{definitionDPP}
Let $X$ be a Stein manifold and $\phi$ be a psh exhaustion function satisfying Condition $\mathfrak{A}$. Choose any smooth volume form $dV$ and corresponding locally integrable global unitary frame $e$ of $K_X$. For each $k>0$, define the determinantal point process $\Theta_k$ with kernel $K^{dV,e}_{k\phi}$ with respect to the measure $e^{-k\phi}dV$ on $X$. They are well-defined by Theorem \ref{MacSo} and Lemma \ref{lemma1}. By \eqref{choice1} and \eqref{choice2}, $\Theta_k$ is independent of the choice of $dV$ and $e$, and only depends on $\phi$.
\end{definition}

\subsection{Pluripotential theory on Stein manifolds}

We begin with the main pluripotential theoretic result in this paper. It shows that the Condition $\mathfrak{A}$ is the right condition anticipated in \cite[Remark 2]{eum2025limiting}. From now on, we always assume $\phi$ satisfies the Condition $\mathfrak{A}$. For a real $u\in C^\infty(X)$, denote the psh envelope by
\begin{equation*}
    P(\phi+u):=\left(\sup\{v\in PSH(X): v\leq \phi+u \}\right)^*.
\end{equation*}

\begin{theorem}\label{plurimain}
For every $u\in C^\infty_c(X)$,
\begin{equation*}
    P(\phi+u)=\phi \quad \text{outside a compact subset of}\; X.
\end{equation*}
That is, $P(\phi+u)-\phi$ is compactly supported. The compact set can be chosen uniformly when $u$ ranges over a bounded subset of a fixed finite dimensional subspace of $C^\infty_c(X)$. Moreover, 
\begin{equation*}
    P(\phi+u)\in C_{\mathrm{loc}}^{1,1}(X).
\end{equation*}
The $C^{1,1}$ norm of $P(\phi+u)-\phi$ can be chosen uniformly as well.
\end{theorem}

\begin{proof}
\emph{(1) Proof of the compact support.---}
Choose a smooth nonincreasing function $\eta:\mathbb{R}\rightarrow [0,1]$ with $\eta=1$ on $(-\infty,1]$ and $\eta=0$ on $[2,\infty)$. Let $C:=\|u\|_\infty$. Choose $L$ large enough so that $\operatorname{supp} u \subset \{\phi<L\}$ and 
\begin{equation}\label{Lcondition}
    \frac{3BC}{L}\|\eta''\|_\infty <1.
\end{equation}

Now set $h:=-C\eta(\phi/L)$. On $\operatorname{supp} u$, $\phi<L$ and we have $h=-C \leq u$, while outside $\operatorname{supp} u$ we have $h\leq 0=u$. Hence $h\leq u$ on $X$. We claim that $\phi+h$ is strictly plurisubharmonic. Indeed,
\begin{equation*}
    i\partial\bar\partial (\phi+h)=\left(1-\frac{C}{L}\eta'(\phi/L)\right)  i\partial\bar\partial \phi-\frac{C}{L^2}\eta''(\phi/L)i\partial \phi\wedge \bar\partial \phi.
\end{equation*}
Since $h$ is nonconstant only on $L\leq \phi\leq 2L$, the second term is bounded below by (using \eqref{Lcondition})
\begin{equation*}
-\frac{C}{L^2}\|\eta''\|_\infty B\phi \,i\partial\bar\partial \phi \geq -\frac{1}{3L}\phi \,i\partial\bar\partial \phi \geq -\frac{2}{3}i\partial\bar\partial \phi.
\end{equation*}
This proves $ i\partial\bar\partial (\phi+h)>0$. Hence $\phi+h$ is plurisubharmonic and lies below $\phi+u$, and we have
\begin{equation*}
    \phi+h \leq P(\phi+u) \leq \phi+u,
\end{equation*}
since $\phi+u$ is continuous. On $\{\phi\geq 2L\}$, both $h$ and $u$ vanish and hence $\operatorname{supp} P(\phi+u)-\phi \subset \{\phi \leq 2L\}$. For the uniform statement, replace $C$ by a common sup-norm bound.

\vspace{2mm}
\emph{(2) Proof of the $C^{1,1}$-regularity.---}
We basically adapt the method of Bedford and Taylor on unit ball case \cite{MR445006}. See also \cite[Section 4.3]{klimek1991pluripotential}. It was adapted on $\mathbb{C}^n$ \cite{berman2009poly} and on projective manifolds \cite{berman2009bergman}. The key point of these arguments is using holomorphic vector fields that generate small translations on manifolds in every direction at each point. In our Stein manifold setting, Cartan's Theorem A provides such vector fields. 

\begin{lemma}\label{lemmatransl}
Let $K$ be a compact subset of $X$. Then there are global holomorphic vector fields $V_1,\cdots,V_m$, neighborhoods $K\Subset U_0 \Subset U_1 \Subset X$, a constant $\epsilon>0$, and a map
\begin{equation*}
    S : U_1 \times B_{\mathbb{C}^m}(0,\epsilon) \longrightarrow X
\end{equation*}
such that
\begin{enumerate}[label=\textup{(\roman*)}]
\item $S$ is jointly holomorphic; in particular, $S_\zeta(x):=S(x,\zeta)$ is holomorphic in $x$;
\item $S_0(x)=x$;
\item $D_\zeta S(x,0): \mathbb{C}^m \rightarrow T_xX$ is surjective for $x$ in a neighborhood of $\overline{U}_0$;
\item in any fixed finite coordinate cover of $\overline{U}_0$,
\begin{equation}\label{Slinear}
 S_\zeta(x)=x+A_x\zeta+O(|\zeta|^2), \quad A_x(e_j)=V_j(x),
\end{equation}
uniformly in $x$.
\end{enumerate}
Moreover, after shrinking the neighborhoods, every pair of sufficiently nearby points $x,y \in U_0$ admits $\zeta$ with $S_\zeta(x)=y$ and $|\zeta| \leq C \textrm{dist}(x,y)$.
\end{lemma}

Let $q=\phi+u$ and $p=P(q)$. By the previous result, there is a compact set $C$ such that $\operatorname{supp}(p-\phi)\cup \operatorname{supp}u\subset C$. Thus $p=q=\phi$ on $X\setminus C$. Choose a relatively compact neighborhood $D$ of $C$ and apply Lemma \ref{lemmatransl} with $K=\overline{D}$. Let $U=U_1$ in the Lemma, and choose a collar $W$ around $\partial D$ satisfying $\partial D \subset W, \;\overline{W}\Subset U\setminus C$. Possibly shrink $\epsilon$ in the Lemma so that for all $\zeta$ in the ball,
\begin{equation*}
    S_{\pm\zeta}(\overline{D})\Subset U,\qquad S_{\pm\zeta}(\overline{W})\Subset X\setminus C.
\end{equation*}

We first show that $p$ is locally Lipschitz. Since $q$ and $\phi$ are smooth on $\overline{U}$, by \eqref{Slinear},
\begin{equation}\label{qphi}
    q(S_{\pm\zeta}x)-C_1|\zeta| \leq q(x), \quad x\in D\quad\text{and}\quad \phi(S_{\pm\zeta}x)-C_1|\zeta| \leq \phi(x), \quad x\in W
\end{equation}
for some $C_1$. On $D$, define the psh function
\begin{equation*}
    r_\zeta(x)=p(S_\zeta x)-C_1|\zeta|.
\end{equation*}
By \eqref{qphi}, $r_\zeta \leq q$ on $D$. On $W\cap D$, both $x$ and $S_\zeta x$ lie in $X\setminus C$, where $p=q=\phi$. Thus \eqref{qphi} implies $r_\zeta \leq p$ on $W\cap D$. It follows that
\begin{equation*}
\widetilde{p}_\zeta=
\begin{cases}
\max\{p,r_\zeta\},&\text{on}\;D,\\
p,&\text{on}\;X\setminus D,
\end{cases}
\end{equation*}
is a psh function that lies below $q$; hence $\widetilde{p}_\zeta \leq p$. Thus we have $r_\zeta \leq p$ on $D$, which gives
\begin{equation*}
    p(S_\zeta x)-p(x)\leq C_1 |\zeta|.
\end{equation*}
For $x,y$ in a small coordinate chart in $D$, by the last assertion of Lemma \ref{lemmatransl}, we have $|\zeta|=O(\textrm{dist}(x,y))$ sending $x$ to $y$ and $y$ to $x$. Hence we get
\begin{equation}\label{Lipshitz}
    |p(x)-p(y)|\leq C_2 \textrm{dist}(x,y).
\end{equation}

We next prove semiconcavity of $p$. Since $q$ is smooth, by \eqref{Slinear},
\begin{equation}
    \frac{q(S_\zeta x)+q(S_{-\zeta}x)}{2}-C_3|\zeta|^2 \leq q(x),\quad x\in D,
\end{equation}
for some $C_3$. The analogous estimate holds for $\phi$ on $W$. On $D$, define the psh function
\begin{equation*}
    R_\zeta(x)=\frac{p(S_\zeta x)+p(S_{-\zeta}x)}{2}-C_3|\zeta|^2.
\end{equation*}
As in the preceding argument, this can be pasted with $p$ and gives
\begin{equation}\label{concave}
    p(S_\zeta x)+p(S_{-\zeta}x)-2p(x)\leq 2C_3 |\zeta|^2.
\end{equation}
Fix a coordinate ball around $x$. For small enough $h\in\mathbb{C}^n$, the last assertion of Lemma \ref{lemmatransl} gives $\zeta=\zeta(x,h)$ with
\begin{equation*}
    S_\zeta x=x+h, \quad |\zeta|=O(|h|).
\end{equation*}
By \eqref{Slinear}, $h=A_x\zeta+O(|\zeta|^2)$. This implies
\begin{equation*}
    S_{-\zeta}x=x-A_x\zeta+O(|\zeta|^2)=x-h+O(|\zeta|^2).
\end{equation*}
By the local Lipschitz estimate \eqref{Lipshitz}, we have $p(S_{-\zeta}x)=p(x-h)+O(|h|^2)$. The estimate \eqref{concave} now gives
\begin{equation}\label{differnce}
     p(x+h)+p(x-h)-2p(x)\leq C_4 |h|^2.
\end{equation}

By a standard change of variable argument, for any real unit vector $e$, from \eqref{differnce} we get $C_4d\lambda_{\mathbb{R}^{2n}}-D^2_{e,e}p \geq 0$ in the sense of distributions. Thus $D^2_{e,e}p$ is a signed Radon measure. On the other hand, plurisubharmonicity gives
\begin{equation*}
    D^2_{e,e}p+D^2_{Je,Je}p \geq 0
\end{equation*}
in the sense of distributions. Applying the upper bound to $Je$, we obtain $D^2_{e,e}p \geq -C_4d\lambda_{\mathbb{R}^{2n}}$. By the Radon--Nikodym theorem applied to $2C_4d\lambda_{\mathbb{R}^{2n}}\geq C_4d\lambda_{\mathbb{R}^{2n}}-D^2_{e,e}p \geq 0$, we conclude that $D^2_{e,e}p$ is represented by an $L^\infty_{\mathrm{loc}}$ function. Then by polarization, every entry of the distributional Hessian is an $L^\infty_{\mathrm{loc}}$ function. That is, $p \in W^{2,\infty}_{\mathrm{loc}}=C^{1,1}_{\mathrm{loc}}$. The uniform statement also follows.
\end{proof}

\begin{proof}[Proof of the Lemma \ref{lemmatransl}]
By Cartan's Theorem A, $TX$ is globally generated. Given a compact set $K$, choose finitely many vector fields $V_1,\cdots,V_m$ spanning $TX$ on a neighborhood of $K$. Choose nested neighborhoods $K\Subset U_0 \Subset U_1 \Subset U_2$. Denote by $\Phi_j^w$ the complex flows of $V_j$ defined on $U_2$, for uniformly small $|w|<\epsilon$. Choose $\epsilon$ such that every possible composition of flows starting from $U_1$ remains in $U_2$. Define
\begin{equation*}
    S=\Phi_m^{\zeta_m}\circ \cdots \circ \Phi_1^{\zeta_1}.
\end{equation*}
It is easy to show that $S$ satisfies (i)--(iv). The last assertion follows from the inverse function theorem.
\end{proof}

Let $q=\phi+u$, $p=P(q)$, and contact set $D=\{p=q\}$. By the classical result of Bedford and Taylor \cite{MR445006,BT82}, a locally bounded psh envelope is maximal on the open noncontact set. Here we recall the argument in the present setting for the sake of completeness. Let $G \Subset \{p<q\}$ and let $v \in PSH(G)$ satisfy $\limsup_{z\rightarrow\partial G}v(z)\leq p(z)$. By continuity, $\delta:=\min_{\overline{G}}(q-p)>0$, and the maximum principle implies that $M:=\sup_G(v-p)_+$ is finite. If $M>0$, choose $0<\varepsilon < \min\{1,\delta/M\}$. Then $(1-\varepsilon)p+\varepsilon v\leq q$ on $G$. Taking the maximum of this function with $p$ on $G$ and extending by $p$ outside $G$, we obtain a global psh competitor below $q$. The definition of the envelope gives $(1-\varepsilon)p+\varepsilon v\leq p$; hence $v\leq p$. The case $M=0$ is immediate. Thus $p$ is maximal on $\{p<q\}$, and therefore
\begin{equation*}
    \mathbf{1}_{\{p<q\}}(dd^c p)^n=0.
\end{equation*}
Denote the Monge--Amp\`ere measure of a locally bounded psh function by 
\begin{equation*}
    \operatorname{MA}(a):=\frac{1}{n!}(dd^c a)^n=\frac{1}{(2\pi)^n n!}(i\partial\bar\partial a)^n.
\end{equation*}
Thus we obtain the \textit{orthogonality relation} (cf. \cite[proof of Lemma 4.11]{berman2010growth})
\begin{equation}\label{orthogonality}
    \int_X (q-p)\operatorname{MA}(p)=0.
\end{equation}
Since $p \in C^{1,1}_{\mathrm{loc}}$, its Monge--Amp\`ere measure is absolutely continuous w.r.t. Lebesgue measure in coordinates.

Next we prove another main result of this section, the differentiability of the equilibrium energy. For $v \in C^{1,1}_c(X)$ such that $a=\phi+v$ is psh, define the Monge--Amp\`ere energy by
\begin{equation*}
    \mathcal{E}_\phi(a)=\frac{1}{(n+1)!}\sum_{j=0}^n \int_X v(dd^c a)^j\wedge (dd^c \phi)^{n-j}.
\end{equation*}
The following lemma records some well-known properties of the MA energy functional in our setting (cf. \cite[Section 4]{berman2010growth}). Since the integrands have compact support, the integrations by parts below are well defined and have no boundary terms.  
\begin{lemma}\label{aaaaaaa}
Let $a=\phi+v$ and $b=\phi+w$ be $C^{1,1}_{\mathrm{loc}}$ psh functions with $v,w \in C^{1,1}_c(X)$. Then:
\begin{enumerate}[label=\textup{(\roman*)}]
\item for $a_t=b+t(a-b)$, $0\leq t\leq 1$, one has
\begin{equation}\label{energyderivative}
   \frac{d}{dt}\mathcal{E}_\phi(a_t)=\int_X (a-b)\operatorname{MA}(a_t),\qquad 0<t<1,
\end{equation}
with the corresponding one-sided derivatives at $t=0$ and $t=1$;
\item one has the inequalities
\begin{equation}\label{energyinequal}
   \int_X (a-b)\operatorname{MA}(a) \leq \mathcal{E}_\phi(a)-\mathcal{E}_\phi(b) \leq\int_X (a-b)\operatorname{MA}(b).
\end{equation}
\end{enumerate}
\end{lemma}

\begin{proof}
Note that $a-b=v-w$ is compactly supported. It is easy to verify by direct computation that
\begin{equation*}
    \mathcal{E}_\phi(a)-\mathcal{E}_\phi(b)=\frac{1}{(n+1)!}\sum_{j=0}^n \int_X (a-b)(dd^c a)^j\wedge (dd^c b)^{n-j}.
\end{equation*}
Alternatively, note that $\mathcal{E}_\phi$ is essentially the Aubin--Yau functional and recall the cocycle property. The derivative expression \eqref{energyderivative} also can be proved by direct computation (see, for example, \cite[Proposition 9]{eum2025limiting}). Next we show that the derivative in \eqref{energyderivative} is decreasing in $t$. For $0<t<1$,
\begin{align*}
    \frac{d}{dt}\int_X (a-b)\operatorname{MA}(a_t)&=\frac{1}{(n-1)!}\int_X (a-b)dd^c(a-b)\wedge (dd^c a_t)^{n-1}\\
    &=-\frac{1}{(n-1)!}\int_X d(a-b)\wedge d^c (a-b)\wedge (dd^c a_t)^{n-1} \leq 0.
\end{align*}
Since $\mathcal{E}_\phi(a)-\mathcal{E}_\phi(b)=\int_0^1 dt\int_X (a-b)\operatorname{MA}(a_t)$, we get \eqref{energyinequal}.
\end{proof}

Now we prove the analog of Theorem B in \cite{berman2010growth}.
\begin{theorem}\label{diffofequil}
For $u\in C^\infty_c(X)$, $u\mapsto \mathcal{E}_\phi(P(\phi+u))$ is Gateaux differentiable, with derivative given by
\begin{equation}\label{Gderivative}
   \left.\frac{d}{dt}\right|_{t=0}\mathcal{E}_\phi(P(\phi+u+th))=\int_X h \operatorname{MA}(P(\phi+u)),\quad h\in C^\infty_c(X).
\end{equation}
The function $u\mapsto \mathcal{E}_\phi(P(\phi+u))$ is also finite and concave.
\end{theorem}

\begin{proof}
Let $q=\phi+u$, $p=P(\phi+u)$, and $p_t=P(q+th)$. We first use the definition of the envelope to show that the envelope operator is a contraction in the sup norm:
\begin{equation}\label{envelopecontraction}
   \|p_t-p\|_\infty \leq |t|\|h\|_\infty.
\end{equation}
Indeed, if $M=\|h\|_\infty$, then $p_t-|t|M \leq q$ and $p-|t|M\leq q+th$. The functions on the left are competitors for the corresponding envelopes. By Theorem \ref{plurimain}, there is a common compact support for $p_t-\phi$ when $|t|\leq1$.

By \eqref{envelopecontraction}, the potentials converge uniformly and are locally bounded as $t\rightarrow 0$. By the continuity property of the MA operator,
\begin{equation}\label{MAcontinuity}
   \operatorname{MA}(p_t) \longrightarrow \operatorname{MA}(p) \qquad \text{locally weakly on} \,X.
\end{equation}
Thus the integrals against every function in $C_c(X)$ converge.

To compute the derivative for $t>0$, apply \eqref{energyinequal} to $p_t$ and $p$. By the argument for the orthogonality relation,
\begin{equation*}
   p_t=q+th \quad \operatorname{MA}(p_t)\text{-almost everywhere}, \quad p=q \quad \operatorname{MA}(p)\text{-almost everywhere}.
\end{equation*}
Since $p \leq q$ and $p_t \leq q+th$, it follows that
\begin{equation*}
   p_t-p \geq th \quad \operatorname{MA}(p_t)\text{-almost everywhere}, \quad p_t-p\leq th \quad \operatorname{MA}(p)\text{-almost everywhere}.
\end{equation*}
Consequently,
\begin{equation}\label{diff1}
   \int_X h\operatorname{MA}(p_t) \leq \frac{\mathcal{E}_\phi(p_t)-\mathcal{E}_\phi(p)}{t} \leq \int_X h\operatorname{MA}(p).
\end{equation}
Letting $t\downarrow 0$ and using \eqref{MAcontinuity}, we obtain the right derivative. For $t<0$, division by $t$ reverses the inequalities and gives
\begin{equation*}
   \int_X h\operatorname{MA}(p) \leq\frac{\mathcal E_\phi(p_t)-\mathcal E_\phi(p)}{t} \leq \int_X h\operatorname{MA}(p_t).
\end{equation*}
The left derivative follows by letting $t\uparrow 0$.

We next prove concavity. For arbitrary $u_0,u_1$, let $p_{u_i}=P(\phi+u_i)$. The right inequality in \eqref{energyinequal}, together with the orthogonality relation, gives
\begin{equation*}
   \mathcal{E}_\phi(p_{u_1})-\mathcal{E}_\phi(p_{u_0})\leq\int_X (p_{u_1}-p_{u_0})\operatorname{MA}(p_{u_0})\leq\int_X (u_1-u_0)\operatorname{MA}(p_{u_0}),
\end{equation*}
since $p_{u_0}=\phi+u_0$ almost everywhere for $\operatorname{MA}(p_{u_0})$ and $p_{u_1}\leq\phi+u_1$. Thus $\operatorname{MA}(p_{u_0})$ is a supergradient and $\mathcal{E}_\phi(P\cdot)$ is concave.
\end{proof}

\subsection{Bergman kernel asymptotics}

In this section we prove that the Bergman measure $k^{-n}\beta_{ka}$ \eqref{Bergmanmeasure} converges to the Monge--Amp\`ere measure $\operatorname{MA}(a)$ of the psh weight $a$ as $k\rightarrow \infty$. Note that in our setting of $C^{1,1}$-regular psh weight $a$, we have
\begin{equation*}
    \operatorname{MA}(a)=c_n\pi^{-n}\det(\partial_j\bar\partial_k a)dZ\wedge d\bar{Z} \quad \text{almost everywhere}.
\end{equation*}
Below we closely follow \cite{berman2009poly,berman2009bergman}. 

\begin{proposition}[Local Morse inequalities]\label{localMorse}
Let $\mathcal{A}$ be a family of psh weights of the form $a=\phi+v, v\in C_c^{1,1}(X)$ with uniformly bounded $C^{1,1}$-norms of $v$. For every compact subset $K$ of $X$, there is a constant $C_K$ such that 
\begin{equation}\label{uniformupper}
    \mathbf{1}_K\beta_{ka}\leq C_K k^n\, dV_K, \qquad a\in \mathcal{A},\; k\geq 1
\end{equation}
for any fixed smooth density $dV_K$. Moreover, at every point $x$ where $a$ has second derivatives, we have
\begin{equation}\label{pointwiseupper}
    \limsup_{k\rightarrow\infty} k^{-n}B_{ka}(x) \leq \pi^{-n}\det(\partial_j\bar\partial_k a(x)).
\end{equation}
\end{proposition}

\begin{proof}
Recall the definition and extremal characterization of $B_{ka}$ \eqref{Bk}. Working on a fixed coordinate chart around $x$, the proof of \cite[Lemma 3.1]{berman2009poly} goes without change. Let us reproduce the main argument of the proof. Fix a chart around $x=0$ and consider the first order Taylor expansion
\begin{equation*}
    a(z)=a(0)+2\operatorname{Re}H_x(z)+O(|z|^2),
\end{equation*}
where $H_x(z)=\partial_z a(0)\cdot z$ is a holomorphic function. Then for $f_s dZ=s\in \mathcal{H}(ka)$,
\begin{equation*}
    |f_s(z)|^2e^{-ka(z)}=|f_s(z)e^{-kH_x(z)}|^2e^{-ka(0)+O(k|z|^2)}.
\end{equation*}
Apply the submean value property to the local holomorphic function $f_s(z)e^{-kH_x(z)}$ on a ball of radius $\propto k^{-1/2}$ and we get \eqref{uniformupper}.

For the pointwise upper bound, assume $a$ has second derivatives at $x$ and consider the second order Taylor expansion
\begin{equation*}
    a(z)=a(0)+2\operatorname{Re}H(z)+\sum_{j=1}^n \lambda_j|z^j|^2+o(|z|^2).
\end{equation*}
Here $H(z)$ is a holomorphic function that also absorbs the holomorphic quadratic part of the Hessian of $a$, and we perform a unitary coordinate transform to diagonalize the complex Hessian of $a$. Set
\begin{equation*}
    F_{k,s}(z)=f_s(z/\sqrt k)e^{-kH(z/\sqrt k)},\quad Q(z)=a(z)-a(0)-2\operatorname{Re}H(z).
\end{equation*}
Fix $R>0$ and consider the ball of radius $R/\sqrt{k}$ (included in the fixed coordinate chart for large enough $k$). Then for fixed $R$, 
\begin{equation}\label{rho}
    \rho_{k,R}:=\exp\left(\sup_{|z|\leq R}\left|kQ\left(\frac{z}{\sqrt{k}}\right)-\sum_j \lambda_j|z^j|^2\right|\right) \longrightarrow 1
\end{equation}
as $k\rightarrow \infty$. Note that we have
\begin{align}
    k^{-n}\frac{|f_s(0)|^2e^{-ka(0)}}{\|s\|^2_{ka}}&\leq k^{-n}\frac{|f_s(0)|^2e^{-ka(0)}}{\int_{|z|\leq R/\sqrt{k}}|f_s|^2 e^{-ka}d\lambda_{\mathbb{R}^{2n}}}\nonumber\\
    &= \frac{|f_s(0)|^2e^{-ka(0)}}{\int_{|z|\leq R}|f_s(\cdot/\sqrt{k})|^2 e^{-ka(\cdot/\sqrt{k})}d\lambda}\leq \rho_{k,R}\frac{| F_{k,s}(0)|^2}{\int_{|z|\leq R}| F_{k,s}|^2 e^{-\sum \lambda_j|z^j|^2}d\lambda}.\label{berman}
\end{align}
Now we provide an upper bound for the last fraction term in \eqref{berman}. Let $r_1,\cdots,r_n>0$ so that $r_1^2+\cdots+r_n^2<R^2$. By the mean value property of holomorphic functions, we have
\begin{equation*}
    |F_{k,s}(0)|^2=\left| \frac{1}{(2\pi)^n}\int_{[0,2\pi]^n} F_{k,s}(r_1e^{i\theta_1},\cdots,r_ne^{i\theta_n})d\theta \right|^2\leq \frac{1}{(2\pi)^n}\int_{[0,2\pi]^n}\left| F_{k,s}(r_1e^{i\theta_1},\cdots,r_ne^{i\theta_n})\right|^2 d\theta.
\end{equation*}
Multiplying it by $e^{-\sum \lambda_j r_j^2}$ and integrating it with respect to $r_jdr_j$ over all $r_1^2+\cdots+r_n^2<R^2$, we get
\begin{equation*}
    |F_{k,s}(0)|^2\int_{|z|\leq R} e^{-\sum \lambda_j |z^j|^2} d\lambda \leq \int_{|z|\leq R} \left| F_{k,s}\right|^2 e^{-\sum \lambda_j |z^j|^2} d\lambda.
\end{equation*}
This gives for fixed $R$,
\begin{equation*}
    \limsup_{k\rightarrow\infty}k^{-n}B_{ka}(0)\leq \frac{1}{\int_{|z|\leq R} e^{-\sum \lambda_j |z^j|^2} d\lambda}.
\end{equation*}
Letting $R$ tend to infinity proves \eqref{pointwiseupper}.
\end{proof}

Next we prove the lower bound of the Bergman kernel function $B_{ka}$, proving the convergence of the Bergman measure to the Monge--Amp\`ere measure. Note that we do not prove the convergence towards the \textit{equilibrium measure} for non-psh weights, since we will not need them in our proof of the LDP (unlike that of \cite{berman2014determinantal}); see the proof of Theorem \ref{wow}. However, it is clear that such an extension is possible in the current setting.

\begin{lemma}\label{psix}
    For every $x\in X$, there exists a locally integrable upper semicontinuous function $\psi_x \leq 0$, smooth on $X\setminus \{x\}$ and constant outside a coordinate ball, such that
    \begin{equation*}
        \psi_x(z)=(n+1)\log |z-x|^2+O(1),\qquad i\partial\bar\partial\psi_x \geq -C_x \omega_\phi
    \end{equation*}
    for some constant $C_x$.
\end{lemma}

\begin{proof}
    Take $\chi \in C_c^\infty(U)$, equal to $1$ near $x$ with $0\leq \chi \leq 1$, for a coordinate ball $U$ around $x$ and extend $(n+1)\chi(z)\log|z-x|^2$ by zero outside of $U$. Since $\chi$ is compactly supported and $\omega_\phi$ is positive definite, we can find such $C_x$.
\end{proof}

\begin{theorem}\label{thmBergman}
    Let $a\in C^{1,1}_{\mathrm{loc}}(X)$ be a psh function such that $a-\phi \in C^{1,1}_c(X)$. Suppose also that
    \begin{equation*}
        i\partial\bar\partial a \geq \varepsilon \omega_\phi
    \end{equation*}
    for some $\varepsilon>0$. Then, at every point $x$ where $a$ has second derivatives, we have
    \begin{equation}\label{liminf}
        \liminf_{k\rightarrow\infty}k^{-n}B_{ka}(x)\geq \pi^{-n}\det(\partial_j\bar\partial_k a(x)).
    \end{equation}
    This in particular implies that $k^{-n}\beta_{ka}$ vaguely converges to $\operatorname{MA}(a)$ as measures.
\end{theorem}

\begin{proof}
As in Proposition \ref{localMorse}, choose a coordinate ball around $x$ and write
\begin{equation*}
    a(z)=a(0)+2\operatorname{Re}H(z)+\sum_{j=1}^n \lambda_j|z^j|^2+o(|z|^2).
\end{equation*}
Since $a-\varepsilon\phi$ is psh, $\lambda_j>0$ for all $j$. Hence we get (possibly after shrinking the coordinate ball)
\begin{equation}\label{quadratic}
    Q_x(z):=a(z)-a(0)-2\operatorname{Re}H(z) \geq c_0 |z|^2
\end{equation}
for some $c_0>0$. Choose $\chi\in C_c^\infty$ supported in the coordinate ball with $0\leq \chi\leq 1$, equal to $1$ near $x$. Let 
\begin{equation*}
    \sigma_k=\chi e^{kH}dZ, \qquad g_k=\bar\partial \sigma_k.
\end{equation*}
Here $\sigma_k$ is extended by zero to all of $X$. Put $A=\operatorname{supp}(\bar\partial\chi)$. By \eqref{quadratic}, 
\begin{equation*}
    \eta:=\min_{A} Q_x>0.
\end{equation*}

Let $\psi_x$ be given by Lemma \ref{psix} and set $\Psi_k=ka+\psi_x$. Recall that $\omega_\phi$ is a complete K\"ahler metric (Proposition \ref{complete}). Thus the weight $\Psi_k$ satisfies the condition of Demailly's $L^2$-estimate for large enough $k$ since
\begin{equation*}
    i\partial\bar\partial \Psi_k \geq (k\varepsilon-C_x)\omega_\phi.
\end{equation*}
Choose $k_0$ so that $k_0 \geq (C_x+1)/\varepsilon$. Then for $k\geq k_0$, we can apply Theorem \ref{demailly} with $\delta=1$. Since $g_k$ is supported on $A$, we have
\begin{equation*}
    \int_X |g_k|^2_{\omega_\phi}e^{-\Psi_k}\frac{\omega_\phi^n}{n!}=e^{-ka(x)}\int_A \left| \bar\partial\chi dZ\right|^2_{\omega_\phi} e^{-kQ_x-\psi_x}\frac{\omega_\phi^n}{n!}\leq Ce^{-ka(x)-k\eta}<\infty.
\end{equation*}
By Theorem \ref{demailly}, we have a $(n,0)$-form $v_k$ with $\bar\partial v_k=g_k$ and
\begin{equation}\label{l2bound}
    \|v_k\|^2_{ka}\leq  \|v_k\|^2_{\Psi_k} \leq Ce^{-ka(x)-k\eta}.
\end{equation}
The first inequality uses $\psi_x\leq 0$.

We claim that $v_k(x)=0$. Indeed, on a ball around $x$ where $\chi\equiv 1$, $g_k=0$ and $v_k$ is holomorphic. Write $v_k=f_kdZ$. Since $a$ is bounded near $x$ and $e^{-\psi_x(z)}\propto |z|^{-2(n+1)}$, \eqref{l2bound} implies
\begin{equation*}
    \int_{|z|<r}|f_k|^2|z|^{-2(n+1)}d\lambda(z) <\infty.
\end{equation*}
Hence $f_k$ must vanish at $x$ at least of order $2$. It follows that $s_k:=\sigma_k-v_k \in \mathcal{H}(ka)$ and its local coefficient at $x$ is $1$ since $H(0)=0$.

Now we compute the $L^2$-norm of $s_k$. First, by the change of variables $w=\sqrt{k}z$ and \eqref{rho},
\begin{equation}\label{sigmal2}
    \|\sigma_k\|^2_{ka}=e^{-ka(x)}\int_{\mathbb{C}^{n}} \chi^2 e^{-kQ_x}d\lambda=e^{-ka(x)}\frac{\pi^n}{k^n\det{(\partial_j\bar\partial_k a(x))}}(1+o(1)).
\end{equation}
Moreover, \eqref{l2bound} and \eqref{sigmal2} imply
\begin{equation*}
    \frac{\|v_k\|_{ka}}{\|\sigma_k\|_{ka}}=O(k^{n/2}e^{-k\eta/2})=o(1).
\end{equation*}
By the triangle inequality, we obtain
\begin{equation*}
    \|s_k\|^2_{ka} \leq e^{-ka(x)}\frac{\pi^n}{k^n\det{(\partial_j\bar\partial_k a(x))}}(1+o(1)).
\end{equation*}
Since $s_k(x)=1$ in local coordinates, we have \eqref{liminf}. The convergence of Bergman measures now follows from the uniform upper bound on compact sets \eqref{uniformupper}, almost everywhere pointwise density convergence provided by \eqref{pointwiseupper} and \eqref{liminf}, and the dominated convergence theorem.
\end{proof}

\begin{remark}[Infinite dimensionality of $\mathcal{H}(k\phi)$]\label{infinitedimension}
The proof shows that $B_{ka}(x)>0$, hence $\mathcal{H}(ka)$ is nontrivial under some mild positivity condition on the psh weight $a$. In fact one can prove more using the fact that $\phi$ satisfies the Condition $\mathfrak{A}$. Let $0<\alpha<1, (1-\alpha)B<1$. Then for $\psi:=\phi^\alpha$, $i\partial\bar\partial \psi >0$ and $\psi/\phi = \phi^{\alpha-1} \rightarrow 0$ at infinity. Hence for fixed $k$ and $C$, there exists a constant $D$ such that $C\psi< k/2\phi+D$. For each $m\in\mathbb{Z}_{>0}$, define $\nu_m \sim (n+m)\chi\log|z|^2$ as in Lemma \ref{psix}. Choose $C_m$ so that $i\partial\bar\partial (C_m\psi+\nu_m)\geq 0$ and let $\Psi_m:=k/2\phi+C_m\psi+\nu_m$. Then $i\partial\bar\partial \Psi_m \geq k/2\omega_\phi$ and $\Psi_m \leq k\phi+D_m$ for some $D_m$. Now repeating the proof of Theorem \ref{thmBergman} with $\sigma_m=\chi (z^1)^mdZ$, we get an infinite sequence of linearly independent elements of $\mathcal{H}(k\phi)$. Hence $\mathcal{H}(k\phi)$ is infinite dimensional if $\phi$ satisfies the Condition $\mathfrak{A}$. If the psh weight $\phi$ does not satisfy Condition $\mathfrak{A}$, $\mathcal{H}(k\phi)$ can be finite dimensional (for example, consider $X=\mathbb{C}$ and $\phi=\log(1+|z|^2)$).
\end{remark}

\subsection{Proof of the Large deviation principle}

To apply the abstract G\"artner--Ellis theorem, we need to study the logarithmic moment generating function of the DPP $\Theta_k$. By the definition of determinantal point processes, for $g\in C_c(X)$, we have (see, for example, \cite{eum2025limiting})
\begin{align*}
    \mathbb{E}\left[\prod_{x\in\Theta_k}(1+g(x))\right]&=\sum_{\ell=0}^\infty \frac{1}{\ell!}\int_{X^\ell} \det\left[K^{dV,e}_{k\phi}(x_i,x_j)\right]_{i,j}\prod_{j=1}^\ell g(x_j)e^{-k\phi(x_j)}dV(x_j)\\
    &=\det_{U_e(\mathcal{H}(k\phi))}\left(I+U_e\circ T_{k\phi,g}\circ U_e^{-1}\right)=\det_{\mathcal{H}(k\phi)}\left(I+T_{k\phi,g}\right).
\end{align*}
Recall that $T_{k\phi,g}$ is indeed a trace class operator (Lemma \ref{lemma1}). Plugging in $g=e^{-ku}-1$ with $u\in C_c^\infty(X)$, we obtain
\begin{equation}\label{logmomgenformula}
    \mathbb{E}\left[e^{-k\langle u,\Theta_k\rangle}\right]=\det_{\mathcal{H}(k\phi)}\left(I+T_{k\phi,e^{-ku}-1}\right)=:\det_{\mathcal{H}(k\phi)}\left(T_{k\phi,e^{-ku}}\right).
\end{equation}
Since $T_{k\phi,e^{-ku}}\geq e^{-k\|u\|_\infty}I$, it is invertible. Define for $u,v \in C_c(X)$,
\begin{equation*}
    \mathcal{D}_{k,\phi+u}(v):=\det_{\mathcal{H}(k(\phi+u))}T_{k(\phi+u),e^{-kv}}.
\end{equation*}

\begin{lemma}\label{Englis}
    For $u,v \in C_c(X)$, 
    \begin{equation*}
        \mathcal{D}_{k,\phi}(u+v)=\mathcal{D}_{k,\phi}(v)\mathcal{D}_{k,\phi+v}(u).
    \end{equation*}
\end{lemma}

\begin{proof}
Let $J: \mathcal{H}(k\phi)\rightarrow\mathcal{H}(k(\phi+v))$ be a bounded invertible operator defined by the identity map. Put
\begin{equation*}
    A=T_{k\phi,e^{-kv}},\qquad B=T_{k(\phi+v),e^{-ku}}.
\end{equation*}
For $f,g\in \mathcal{H}(k\phi)$,
\begin{align*}
    \langle A J^{-1} B J f,g\rangle_{k\phi}&=\langle M_{e^{-kv}} J^{-1} B J f,g\rangle_{k\phi}=\langle B J f,Jg \rangle_{k(\phi+v)}\\
    &=\langle M_{e^{-ku}}J f, Jg\rangle_{k(\phi+v)} =\langle T_{k\phi,e^{-k(u+v)}}f,g \rangle_{k\phi}.
\end{align*}
    That is, $T_{k\phi,e^{-k(u+v)}}=A J^{-1} B J$. Since Fredholm determinants are multiplicative on the product of operators, we prove the claim.
\end{proof}

\begin{lemma}
    Let $v\in C^{1,1}_c(X)$ and suppose $a_t=\phi+tv$ is psh for $0\leq t\leq 1$. Then
    \begin{equation}\label{derivativeformula}
        \frac{d}{dt}\log \mathcal{D}_{k,\phi}(tv)=-k\int_X v \beta_{ka_t}.
    \end{equation}
\end{lemma}

\begin{proof}
By \cite[Corollary 5.2]{simon2005trace}, we have
\begin{equation*}
    \frac{d}{dt}\log \det T_{k\phi,e^{-ktv}}=\operatorname{Tr}\left(T^{-1}_{k\phi,e^{-ktv}}T_{k\phi,-kve^{-ktv}} \right).
\end{equation*}
Let $J_t:\mathcal{H}(k\phi)\rightarrow\mathcal{H}(ka_t)$ be the identity map and set
\begin{equation*}
    A_t=T_{k\phi,e^{-ktv}}, \quad C_t=T_{k\phi,-kve^{-ktv}}, \quad S_t=T_{ka_t,-kv}.
\end{equation*}
Then the proof of Lemma \ref{Englis} gives $C_t=A_t J_t^{-1} S_t J_t$. Since $A'_t=C_t$, it follows that
\begin{equation*}
    \operatorname{Tr}\left(T^{-1}_{k\phi,e^{-ktv}}T_{k\phi,-kve^{-ktv}} \right)=\operatorname{Tr}\left(A_t^{-1}A'_t\right)=\operatorname{Tr}\left(J_t^{-1} S_t J_t\right)=\operatorname{Tr}\left(T_{ka_t,-kv}\right).
\end{equation*}
The trace formula \eqref{traceformula} then gives \eqref{derivativeformula}.
\end{proof}

Now we can combine our results so far to identify the limit of the logarithmic moment generating functions $\Lambda$ in the abstract G\"artner--Ellis theorem in our setting.

\begin{theorem}\label{wow}
    For every $u\in C_c^\infty(X)$,
    \begin{equation*}
        \lim_{k\rightarrow\infty}\frac{1}{k^{n+1}}\log\mathbb{E}\left[e^{-k\langle u,\Theta_k\rangle}\right]=-\mathcal{E}_\phi\left(P(\phi+u)\right).
    \end{equation*}
\end{theorem}

\begin{proof}
    By \eqref{logmomgenformula}, we need to show that
    \begin{equation*}
         \lim_{k\rightarrow\infty}\frac{1}{k^{n+1}}\log\mathcal{D}_{k,\phi}(u)=-\mathcal{E}_\phi\left(P(\phi+u)\right).
    \end{equation*}
Let $p=P(\phi+u)=\phi+v$, $r=u-v=(\phi+u)-p$. By Theorem \ref{plurimain}, $v\in C^{1,1}_c(X)$ and $r\geq 0$ is compactly supported. Lemma \ref{Englis} gives
\begin{equation}\label{I+II}
    \log \mathcal{D}_{k,\phi}(u)=\log \mathcal{D}_{k,\phi}(v)+\log \mathcal{D}_{k,p}(r).
\end{equation}
    For the first term, consider the path $a_t=\phi+tv, \,t\in [0,1]$. For each $t$, $a_t$ is psh and $a_t-\phi$ has a uniformly bounded $C^{1,1}$-norm. Integrating \eqref{derivativeformula} then gives
    \begin{equation*}
        \frac{1}{k^{n+1}}\log\mathcal{D}_{k,\phi}(v)= \frac{1}{k^{n+1}}\int_0^1 dt\frac{d}{dt}\log \mathcal{D}_{k,\phi}(tv)=-\frac{1}{k^{n}}\int_0^1 dt \int_X v\beta_{ka_t}.
    \end{equation*}
Since $i\partial\bar\partial a_t=(1-t)i\partial\bar\partial\phi +ti\partial\bar\partial p\geq (1-t)\omega_\phi$, we can apply Theorem \ref{thmBergman} for $t<1$ and, using the uniform upper bound \eqref{uniformupper}, obtain
\begin{equation*}
    \lim_{k\rightarrow\infty}-\int_0^1 dt \int_X v\,\frac{1}{k^n}\beta_{ka_t}= -\int_0^1 dt \int_X v\operatorname{MA}(a_t)=-\mathcal{E}_\phi\left(p\right).
\end{equation*}
The last equality uses Lemma \ref{aaaaaaa}. Thus we only need to show that the limit of the second term in \eqref{I+II} is zero. For that purpose we need the following operator theoretic lemma:
\begin{lemma}\label{Schwarz}
    If $p$ is psh and $r\geq 0$ is compactly supported, then
    \begin{equation}\label{oprSchwarz}
        -k\operatorname{Tr}T_{kp,r}\leq \log \mathcal{D}_{k,p}(r)\leq 0.
    \end{equation}
\end{lemma}
\begin{proof}[Proof of Lemma \ref{Schwarz}]
    Let $B=T_{kp,e^{-kr}}$. Then $e^{-k\|r\|_\infty}I\leq B\leq I$, which gives the upper bound in \eqref{oprSchwarz}. Since $B-I$ is trace class and $B$ is strictly positive, by functional calculus $\log B$ is trace class and satisfies $\operatorname{Tr}\log B=\log\det B$. Since $-\log$ is an operator convex function\footnote{Use the integral representation $-\log x=\int_0^\infty\left((x+s)^{-1}-(1+s)^{-1}\right)ds$ and the fact that $x^{-1}$ is operator convex.}, \cite{davis1957schwarz} implies
    \begin{equation*}
        -\log B=-\log(P_{kp}M_{e^{-kr}}) \leq -P_{kp}\log(M_{e^{-kr}})=kT_{kp,r}.
    \end{equation*}
    Taking traces we get the lower bound in \eqref{oprSchwarz}.
\end{proof}
Now we can finish the proof. By \eqref{oprSchwarz}, \eqref{traceformula}, Proposition \ref{localMorse}, and Fatou's lemma, we have
\begin{align*}
    &0\geq \frac{1}{k^{n+1}}\log \mathcal{D}_{k,p}(r) \geq -\frac{1}{k^{n}}\operatorname{Tr}T_{kp,r}=-\frac{1}{k^n}\int_X r\beta_{kp}\,;\\
    &0\leq \limsup_{k\rightarrow\infty}\frac{1}{k^n}\int_X r\beta_{kp} \leq \int_X r \operatorname{MA}(p)=0.
\end{align*}
The last equality follows from the orthogonality relation \eqref{orthogonality}.
\end{proof}

Finally, to use the abstract G\"artner--Ellis theorem, we need to prove exponential tightness of the family. The point process $\Theta_k$ is a Borel random element of $M_+$. Multiplication by $k^{-n}$ is vaguely continuous, so $L_k:=k^{-n}\Theta_k$ is also a Borel random element of $M_+$. We have shown that the inclusion $\iota:M_+ \hookrightarrow (M,\tau)$ is continuous and we put
\begin{equation*}
    \Gamma_k:=\mathrm{Law}(\iota(L_k)), \qquad r_k:=k^{n+1}.
\end{equation*}
These are Borel probability measures on $(M,\tau)$ supported on $M_+$.

\begin{proposition}\label{thmexptight}
The family of probability measures $\{\Gamma_k\}$ on $M$ is exponentially tight at speed $r_k$.
\end{proposition}

\begin{proof}
    Choose compact sets $K_j \Subset X$ of $X$ such that $K_j \subset K_{j+1}^\circ$, $\bigcup_{j\geq 0} K_j=X$. Choose $\chi_j\in C_c^\infty(K_{j+1})$ with $\chi_j\geq 0$ and equal to $1$ on a neighborhood of $K_j$. Set
    \begin{equation}\label{bj}
        b_j:=\sup_{k\geq 1}\frac{1}{r_k}\log \mathbb{E}\exp\left(r_k\langle \chi_j,L_k \rangle\right).
    \end{equation}
    We first show that $b_j$ is finite. For fixed $k$, we have (by \eqref{logmomgenformula})
    \begin{equation*}
        \mathbb{E}\exp\left(r_k\langle \chi_j,L_k \rangle\right)= \mathbb{E}\exp\left(k\langle \chi_j,\Theta_k \rangle\right)=\det_{\mathcal{H}(k\phi)} \left(I+T_{k\phi,e^{k\chi_j}-1}\right)<\infty.
    \end{equation*}
By Theorem \ref{wow}, the normalized logarithms in \eqref{bj} converge to the finite value $-\mathcal{E}_\phi(P(\phi-\chi_j))$. Hence they are bounded and $b_j<\infty$ (they are trivially bounded on finite interval $k\in[1,R]$). Moreover, $\chi_j \geq 0$ and $L_k \geq 0$ imply that $b_j \geq 0$. Therefore, we have
\begin{equation}\label{222}
    \mathbb{E}e^{r_k\langle \chi_j,L_k \rangle} \leq e^{r_kb_j}
\end{equation}
for all $k, j$.

For $\alpha>0$, define
\begin{equation*}
    \mathcal{C}_\alpha:=\{ \mu\in M_+ : \langle \chi_j,\mu\rangle \leq b_j+\alpha+j \;\,\text{for all}\;\, j\geq 0 \}.
\end{equation*}
By Theorem \ref{thmcompact}, each $\mathcal{C}_\alpha$ is a compact subset of $M_+$, hence a compact subset of $M$. By Markov's inequality and \eqref{222}, we have 
\begin{align*}
    \Gamma_k(M\setminus \mathcal{C}_\alpha) &\leq \sum_j \mathbb{P}\left[ \langle \chi_j, L_k \rangle >b_j+\alpha+j\right] = \sum_j \mathbb{P}\left[ e^{r_k\langle \chi_j, L_k \rangle} > e^{r_k(b_j+\alpha+j)}\right]\\
    &\leq \sum_j e^{-r_k(b_j+\alpha+j)}\mathbb{E}e^{r_k\langle \chi_j, L_k \rangle}\leq \sum_j e^{-r_k(\alpha+j)}=\frac{e^{-r_k\alpha}}{1-e^{-r_k}}.
\end{align*}
Consequently,
\begin{equation*}
    \limsup_{k\rightarrow\infty}\frac{1}{r_k}\log \Gamma_k(M\setminus \mathcal{C}_\alpha) \leq -\alpha,
\end{equation*}
and this completes the proof.
\end{proof}

Before we finally state our main theorem, we need to discuss the rate function. Theorem \ref{wow} gives
\begin{equation*}
    \lim_{k\rightarrow\infty}\frac{1}{r_k}\Lambda_k(r_k f)= \lim_{k\rightarrow\infty}\frac{1}{k^{n+1}}\log \mathbb{E} e^{k\langle f,\Theta_k\rangle}= -\mathcal{E}_\phi(P(\phi-f))=:\Lambda(f).
\end{equation*}
By Theorem \ref{diffofequil}, $\Lambda$ is finite and Gateaux differentiable. Together with Proposition \ref{thmexptight}, this allows us to apply the abstract G\"artner--Ellis theorem (Theorem \ref{thmGE}) and we obtain an LDP for $\{\Gamma_k\}$ on $(M,\tau)$, at speed $r_k$, with rate function
\begin{equation}\label{Legendre}
    \Lambda^*(\mu):=\sup_{f\in C_c^\infty(X)}\biggl\{ \int_X fd\mu -\Lambda(f)\biggr\}.
\end{equation}
Consider a measure $\mu\notin M_+$. Then there is a function $0 \leq \chi \in C_c^\infty(X)$ such that $\langle \chi,\mu \rangle <0$. Since $L_k \geq 0$, we have for $t>0$,
\begin{equation*}
    \mathbb{E}e^{-tr_k\langle \chi, L_k \rangle} \leq 1 \implies \Lambda(-t\chi)\leq 0.
\end{equation*}
Using $f=-t\chi$ in \eqref{Legendre} yields
\begin{equation*}
     \Lambda^*(\mu) \geq -t\langle \chi,\mu\rangle-\Lambda(-t\chi) \geq -t\langle \chi,\mu\rangle \longrightarrow \infty
\end{equation*}
as $t\rightarrow \infty$. Thus $\Lambda^*=\infty$ on $M\setminus M_+$. For $\mu\in M_+$, substituting $u=-f$ gives
\begin{equation*}
    I_\phi(\mu):=\Lambda^*(\mu)=\sup_{u\in C_c^\infty(X)}\biggl\{ \mathcal{E}_\phi(P(\phi+u))-\int_X ud\mu \biggr\}.
\end{equation*}
Since the laws are supported on the closed cone $M_+$ and the rate function is infinite outside of it, the LDP restricts to an LDP on $M_+$. It remains to identify the minimizer. We have $\Lambda(0)=0$, and \eqref{Gderivative} implies
\begin{equation*}
    D\Lambda(0)[h]=\int_X h\operatorname{MA}(\phi).
\end{equation*}
Convexity of $\Lambda$ then implies $\Lambda(f)\geq \int_X f \operatorname{MA}(\phi)$, and consequently $I_\phi(\operatorname{MA}(\phi))=0$. Conversely, if $I_\phi(\mu)=0$, then
\begin{equation*}
    t\int_X h d\mu \leq \Lambda(th),\qquad  \forall h\in C_c^\infty(X),  \forall t\in \mathbb{R}.
\end{equation*}
This implies 
\begin{equation*}
    \int_X h d\mu = D\Lambda(0)[h]=\int_X h\operatorname{MA}(\phi),
\end{equation*}
for every compactly supported $h$. Hence $\mu=\operatorname{MA}(\phi)$. This proves the uniqueness of the minimizer and completes the proof of our main theorem.

\begin{theorem}\label{thmmain}
Let $X$ be a Stein manifold with a plurisubharmonic exhaustion function $\phi$ satisfying the Condition $\mathfrak{A}$. Let $\Theta_k$ be the determinantal point process associated with the weighted Bergman space of holomorphic $n$-forms with weight $e^{-k\phi}$, defined in Definition \ref{definitionDPP}. Set 
\begin{equation*}
    L_k=\frac{1}{k^n}\Theta_k.
\end{equation*}
The laws of $L_k$, on the space of positive Radon measures with the vague topology, satisfy a large deviation principle with speed $r_k=k^{n+1}$ and a good rate function
\begin{equation*}
    I_\phi(\mu)=\sup_{u\in C_c^\infty(X)}\biggl\{ \mathcal{E}_\phi(P(\phi+u))-\int_X ud\mu \biggr\}.
\end{equation*}
The unique zero of $I_\phi$ is
\begin{equation*}
    \mu_\phi=\operatorname{MA}(\phi)=\frac{(dd^c \phi)^n}{n!}.
\end{equation*}
\end{theorem}

\section{Outlook}

In this paper, we prove the LDP via an abstract G\"artner--Ellis theorem, with results from pluripotential theory on psh envelopes and Monge--Amp\`ere energy, together with $L^2$-estimates for the $\bar\partial$-equation, fitting together very naturally. This line of proof of the LDP for DPPs was developed by Berman (and Boucksom) through a series of papers \cite{berman2009poly,berman2009bergman,berman2010growth}, culminating in \cite{berman2014determinantal}. It is quite interesting to see this kind of seamless combination of diverse ideas from complex geometry and probability theory. We hope that such connections between different ideas will inspire further study. For example, we suggest the following direction.

While the present paper concerns infinite-dimensional Bergman spaces on noncompact manifolds, it is also natural to seek extensions of the finite-dimensional weighted polynomial setting on $\mathbb{C}^n$. One possible direction is provided by complete K\"ahler manifolds under suitable curvature and growth assumptions, where holomorphic functions of polynomial growth serve as analogues of polynomials. See, for example, \cite{liu2021dimension,liu2023brief}, where ideas such as Gromov--Hausdorff convergence enter the picture. It would be interesting to see how determinantal point processes can be studied in this setting, and what new connections might emerge. We leave this as a direction for future work.

\subsection*{Acknowledgement}
This work was supported by the National Research Foundation of Korea (NRF) grant funded by the Korea government (MSIT) RS-2024-00346651.
\subsection*{Use of generative AI}
The author used ChatGPT 5.6 sol and ChatGPT 6 Astra to assist with the mathematical content of this paper, from initial proof generation to error correction. All AI-assisted output was independently checked and revised by the author, who takes full responsibility for the final manuscript.

\bibliographystyle{alpha}
\bibliography{References}

\end{document}